\documentclass[reqno,12pt]{amsart}        
\usepackage{amsfonts}
\usepackage{colordvi}
\usepackage{amsmath,amssymb,amsthm} 
\usepackage{mathrsfs}
\usepackage{graphicx}
\usepackage{xypic}

\theoremstyle{plain}      
\newtheorem{thm}{Theorem}[section]     
\newtheorem{theorem}[thm]{Theorem}     
     
\newtheorem{lemma}[thm]{Lemma}     
\newtheorem{proposition}[thm]{Proposition}

\theoremstyle{remark}

\theoremstyle{definition}      
\newtheorem{definition}[thm]{Definition}     

\newcommand{\supp}{\mathop{\rm supp}\nolimits}
\newcommand{\Verte}{\mathop{\rm Vert}\nolimits}
\newcommand{\ord}{\mathop{\rm ord}\nolimits}
\newcommand{\Iim}{\mathop{\rm Im}\nolimits}

\newcommand{\R}{{\mathbb R}}

\newcommand{\Log}{\mathop{\rm Log}\nolimits}

\newcommand{\Arg}{\mathop{\rm Arg}\nolimits}

\newcommand{\val}{\mathop{\rm val}\nolimits}

\newcommand{\Ima}{\mathop{\rm Im}\nolimits}

\title[Geometric structure of phase tropical hypersurfaces]{Geometry and a natural  symplectic structure of phase tropical hypersurfaces}

 \author{Young Rock Kim and Mounir Nisse}

\address{Young Rock Kim\\
 Major in Mathematics Education, Graduate School of Education, Hankuk University of Foreign Studies,\\ 
107 Imunro Dongdaemun-gu, Seoul, 130-791, Korea}
\email{rocky777@hufs.ac.kr}

\address{Mounir Nisse\\
 School of Mathematics KIAS\\ 
 87 Hoegiro Dongdaemun-gu\\ 
 Seoul, 130-722, Republic of Korea.}
\email{mounir.nisse@gmail.com}
\thanks{This work was supported by Basic Science Research Program through the National Research Foundation of Korea (NRF) funded by the Ministry of Education (NRF-2015R1D1A1A01059643).}
\subjclass[2010]{14T05, 32A60, 53D40}
\keywords{Polyhedral complex, tropical variety, (co)amoeba, phase tropical hypersurface}

\begin{document}
\maketitle

\begin{abstract}
First, we define phase tropical hypersurfaces in terms of a degeneration data of smooth complex  algebraic hypersurfaces in $(\mathbb{C}^*)^n$. Next,  we prove that complex hyperplanes are diffeomorphic to their degeneration called phase tropical hyperplanes. More generally,  using  Mikhalkin's decomposition into  pairs-of-pants of smooth algebraic hypersurfaces, we  show that phase tropical hypersurfaces with smooth tropicalization,   possess   naturally a smooth differentiable structure. Moreover, we prove that phase tropical hypersurfaces possess a  natural symplectic structure. 
\end{abstract}

\section{introduction}

In this paper we deal with smooth algebraic hypersurfaces in the complex projective space $\mathbb{CP}^n$. So, let $V$ be  a smooth hypersurface  in $\mathbb{CP}^n$ of degree $d$. Recall that  for a fixed degree, generically a hypersurface in the projective space is smooth and  transverse to all coordinate hyperplanes and all their intersections. Moreover, hypersurfaces in $\mathbb{CP}^n$ with the same degree are  all  diffeomorphic,  and if we equip these hypersurfaces  with the Fubini-Study symplectic form on $\mathbb{CP}^n$ then they are also symplectomorphic.  We denote by $\mathring{V}$ the intersection $V\cap (\mathbb{C}^*)^n$ where $(\mathbb{C}^*)^n$ is the complement of the coordinate hyperplanes in  $\mathbb{CP}^n$. In this case, $\mathring{V}$ is given by some polynomial equation.  One can degenerate the complex standard structure of the complex algebraic torus   to a worst possible degeneration, called ``maximal degeneration'' by M. Kontsevich and Y. Soibelman (see \cite{KS-00} and \cite{KS-04}), and see what kind of geometry can have a degeneration of our variety  $\mathring{V}$. 
After taking the logarithm, $(\mathbb{C}^*)^n$   degenerates or, in other words, collapse onto $\mathbb{R}^n$,  and our hypersurface onto  a balanced rational polyhedral complex $\Gamma$ called {\em{tropical variety}}. One can ask the following question:{\em{ What kind of geometry one can have on a nice lifting in $(\mathbb{C}^*)^n$  of this  balanced rational polyhedral complex?}} This paper give an answer to this question using tools from tropical  and phase tropical geometry.

\vspace{0.2cm}

Tropical geometry is a recent area of mathematics that can be seen as a limiting aspect (or ``degeneration'') of algebraic geometry. Where complex curves viewed as Riemann surfaces turn to metric graphs (one dimensional combinatorial object), and $n$-dimensional complex varieties turn to $n$-dimensional polyhedral complexes with some properties such as  the balancing condition. In other words, tropical varieties are finite dimensional polyhedral complexes with some additional properties. As example, the tropical projective space $\mathbb{TP}^1$ is a smooth projective tropical variety homeomorphic to the  segment. In general, the tropical projective space $\mathbb{TP}^n$ is a smooth projective tropical variety homeomorphic to the  $n$-dimensional simplex. Moreover, as in the classical algebraic geometry, a projective tropical $n$-variety $V$ is a certain $n$-dimensional polyhedral complex in $\mathbb{TP}^N$. One of the most interesting projective tropical varieties are obtained by the tropical limit of a family of projective  algebraic varieties $V_t$ with $1\leq t<\infty$ and $t$ tends to $\infty$. To be more precise, they are the limit of amoebas where  amoebas of algebraic (or analytic) varieties are their image under the logarithm with base a real number $t$. For example, every tropical hypersurface is provided by such way. Tropical objects are some how, the image of a classical objects under the logarithm with base infinity, they are  also called  non-Archimedean amoebas. 

\vspace{0.2cm}
 Phase tropical varieties are some  lifting of tropical  varieties in the complex algebraic torus.  More precisely, for 
  any  strictly positive real  number $t$ we define the self diffeomorphism $H_t$ of $(\mathbb{C}^*)^n$. This defines a new  complex structure $J_t$ on $(\mathbb{C}^*)^n$ denoted by $J_t$ different  from  the standard complex structure if $t\ne e^{-1}$.  One way  to define  phase tropical varieties,   is to take the limit  $\mathring{V}_\infty$ (with respect to the Hausdorff metric on compact sets in $(\mathbb{C}^*)^n$) of a family of  $J_t$-holomorphic varieties $\{\mathring{V}_t\}_{t\in [e^{-1}, \infty)}$ when $t$ goes to $\infty$.  First, in case of hypersurfaces, we prove that  if the hypersurfaces $\mathring{V}_t$  are smooth with same degree (i.e. their defining polynomials have the same Newton polytope $\Delta$), then for a sufficently large $t$ the  $\mathring{V}_t$'s are diffeomorphic to their degeneration $\mathring{V}_\infty$, and the compactification $M_\infty$ of $\mathring{V}_\infty$ in the toric variety $X_\Delta$ associated to $\Delta$ (see Subsection 3.3 for the precise definition of $X_\Delta$) have the same properties, and we have the following:
 
\begin{theorem}\label{A}
Let $\mathring{V}_t\subset (\mathbb{C}^*)^n$ be a family of  smooth complex algebraic hypersurfaces with a fixed degree $\Delta$, and denote by $\mathring{V}_\infty$ the phase tropical hypersurface associated to the family  $\{\mathring{V}_t\}_t$ (i.e., the limit of $H_t(\mathring{V}_t)$ when $t$ goes to $\infty$). Then for a sufficiently large $t\gg 0$ the following statements hold:
\begin{itemize}
\item[(i)]\, The hypersurface $\mathring{V}_t$ is diffeomorphic to $\mathring{V}_\infty$;
\item[(ii)]\, The  compactification $M_\infty$ of $\mathring{V}_\infty$ in the toric variety $X_\Delta$ associated to $\Delta$ is diffeomorphic to $V_t$, where $V_t$ is the closure of $\mathring{V}_t$ in $X_\Delta$.
\end{itemize} 
\end{theorem}

\vspace{0.2cm}

Moreover, using the fact that  pairs-of-pants possess a natural symplectic structure which gives rise to  the standard symplectic structure on the complex projective space  $\mathbb{CP}^n$ after compactification (i.e. collapsing the pair-of-pants boundary), and the gluing of pairs-of-pants can be done in a natural way symplectically,  we obtain a natural symplectic structure on all our phase tropical hypersurface.

Let $(\mathring{V}_t, \iota_t^*(\omega))\subset ((\mathbb{C}^*)^n, \omega)$   be  a family of smooth symplectic hypersurfaces where $\iota_t$ is the inclusion map $\iota_t: \mathring{V}_t\hookrightarrow (\mathbb{C}^*)^n$,
and $\omega$ is the symplectic form on the complex algebraic torus $(\mathbb{C}^*)^n$ defined by:
\begin{equation}\label{(1)}
\omega = \frac{1}{2\sqrt{-1}}\sum_{i=1}^n\frac{dz_i}{z_i}\wedge\frac{d\overline{z}_i}{\overline{z}_i}.
\end{equation}
 Moreover, assume that the  phase tropical hypersurface  $\mathring{V}_\infty$  which the limit  (with respect to the Hausdorff metric on compact sets in $(\mathbb{C}^*)^n$)  exists
 and is equipped with the natural symplectic structure (i.e., with the natural symplectic form $\omega_{nat}$) constructed  by Theorem  \ref{B}.   One can ask the following natural question:  {\em{Are $(\mathring{V}_t, \iota_t^*(\omega))$ and $(\mathring{V}_\infty, \omega_{nat})$ symplectomorphic?}}
 
 \vspace{0.2cm}
 
 The following theorem gives an affirmative answer to this question:

\vspace{0.2cm}

\begin{theorem}\label{B}
Let $\mathring{V}_t\subset (\mathbb{C}^*)^n$ be a family of  smooth complex algebraic hypersurfaces with a fixed degree  $\Delta$, and denote by $\mathring{V}_\infty$ the phase tropical hypersurface associated to the family  $\{\mathring{V}_t\}_t$ (i.e., the limit of $H_t(\mathring{V}_t)$ when $t$ goes to $\infty$). With notations as above, and for a sufficiently large $t\gg 0$ the following statements hold:
\begin{itemize}
\item[(i)]\,  The hypersurface $\mathring{V}_\infty$ possesses  a natural smooth symplectic structure;
\item[(ii)]\,  the hypersurfaces $(\mathring{V}_t, \iota_t^*(\omega))$ and $(\mathring{V}_\infty, \omega_{nat})$ are symplectomorphic.
\end{itemize}
\end{theorem}

\vspace{0.2cm}

We will use the natural  logarithm i.e. with base the  Napier's constant $e$, so that the Archimedean amoeba of
a subvariety of the complex torus $(\mathbb C^*)^n$ is its image under the coordinatewise logarithm map. Recall that amoebas were introduced by Gelfand, Kapranov, and Zelevinsky in 1994 \cite{GKZ-94}. The coamoeba of a subvariety of $(\mathbb C^*)^n$ is its image under the coordinatewise argument map to the real torus $(S^1)^n$.
Coamoebas were introduced by Passare in a talk in 2004 (see \cite{NS-11} and \cite{NS-13} for more details about coamoebas).

\vspace{0.2cm}  
 
This paper is organized as follows. In Section 2, we explain preliminary results in this area. In Section 3, we define phase tropical hypersurface and describe tropical localization. In Section 4,  we describe examples  of coamoebas and  phase tropical hypersurfaces. In Section 5, we give the proof of Theorem \ref{A}. In Section 6, we construct in a natural way a symplectic structure on phase tropical varieties which proves Theorem \ref{B}.

\section{Preliminaries}
In this section we  recall basic concepts of tropical hyperurfaces relevant for our paper. For the general case  we can see \cite{MS-15} with more details. We consider algebraic hypersurfaces $\mathring{V}$ in the complex algebraic torus $(\mathbb{C}^*)^n$, where $\mathbb{C}^*=\mathbb{C}\setminus \{ 0\}$ and $n\geq 1$ an integer. This means that $\mathring{V}$ is the zero locus of a polynomial:
\begin{equation}\label{(1)}
f(z) =\sum_{\alpha\in \supp (f)} a_{\alpha}z^{\alpha}, \quad
z^{\alpha}=z_1^{\alpha_1}z_2^{\alpha_2}\ldots z_n^{\alpha_n},
\end{equation}
where each $a_{\alpha}$ is a non-zero complex number and $\supp (f)$ is a
finite subset of $\mathbb{Z}^n$, called the support of the polynomial
$f$,  and its  convex hull in $\mathbb{R}^n$ is called the Newton polytope
 of $f$ that we denote by $\Delta_f$.
Moreover, we assume that $\supp (f)\subset
\mathbb{N}^n$ and $f$ has no factor of the form $z^{\alpha}$. 

\vspace{0.3cm}

The {\em amoeba} $\mathscr{A}_f$ of an algebraic variety $\mathring{V}\subset (\mathbb{C}^*)^n$ is by definition (see M. Gelfand, M.M. Kapranov and A.V. Zelevinsky \cite{GKZ-94}) the image of $\mathring{V}$ under the map :
\[
\begin{array}{ccccl}
\Log&:&(\mathbb{C}^*)^n&\longrightarrow&\mathbb{R}^n\\
&&(z_1,\ldots ,z_n)&\longmapsto&(\log |z_1| ,\ldots ,\log |
z_n|).
\end{array}
\]
%
\vspace{0.2cm}

Let $\mathbb{K}$ be the field of  Puiseux series with real powers, which is the field of series $\displaystyle{a(t) =\sum_{j\in A_a}\xi_jt^j}$ with $\xi_j\in \mathbb{C}^*$ and $A_a$ is a well-ordered subset of $\mathbb{R}$ (it means any of its subsets has a smallest element). It is well known that the field $\mathbb{K}$ is algebraically closed of characteristic zero. Moreover, it has a non-Archimedean valuation $\val (a) = - \min A_a$:
\[ 
\left\{ \begin{array}{ccc}
\val (ab)&=& \val (a) + \val (b) \\
\val (a + b)& \leq& \max \{ \val (a)  ,\, \val (b)  \} ,
\end{array}
\right.
\]
and we set $\val (0) = -\infty$. Let $g\in \mathbb{K}[z_1,\ldots ,z_n]$ be a polynomial as in \eqref{(1)}. 
If $<,>$ denotes the scalar product in $\mathbb{R}^n$, then the following  piecewise affine linear convex function $\displaystyle{g_{trop}(x) = \max_{\alpha\in\supp (g) } \{ \val (a_{\alpha}) + <\alpha , x>  \}}$, which is in the same time the Legendre transform of the function $\nu :\supp (g)\rightarrow \mathbb{R}$ defined by $\nu (\alpha ) = \min A_{a_{\alpha}}$, is called the {\em tropical polynomial} associated to $g$.

\begin{definition} The tropical hypersurface $\Gamma_g$ is the set of points in $\mathbb{R}^n$ where the tropical polynomial 
 $g_{trop}$ is not smooth (called the corner locus of  $g_{trop}$).
\end{definition}

\noindent We have the following Kapranov's theorem (see \cite{K-00}):

\begin{theorem}[\cite{K-00}, Kapranov] The tropical hypersurface $\Gamma_g$ defined by the tropical polynomial $g_{trop}$ is the subset of $\mathbb{R}^n$ image under  the valuation map of the algebraic hypersurface  with defining polynomial $g$. 
\end{theorem}

\noindent $\Gamma_g$ is also called the non-Archimedean amoeba of the zero locus of $g$ in $(\mathbb{K}^*)^n$.

\vspace{0.1cm} 

Let $g$ be a polynomial as above, $\Delta$ its Newton polytope, and $\tilde{\Delta}$ its  extending Newton polytope, i.e.,
 $\tilde{\Delta} := \mbox{convexhull} \{ (\alpha , r )\in \supp (g)\times \mathbb{R} \mid \, r\geq \min A_{a_{\alpha}} \}$. Let us extend the above function $\nu$ (defined on $\supp (g)$) to all $\Delta$ as follow:
\[
\begin{array}{ccccl}
\nu&:&\Delta&\longrightarrow&\mathbb{R}\\
&&\alpha&\longmapsto&\min \{ r\mid\, (\alpha ,r)\in \tilde{\Delta} \}.
\end{array}
\]
By taking the linear subsets of the lower boundary of $\tilde{\Delta}$,  it is clear that the linearity domains of $\nu$ define a convex subdivision $\tau = \{\Delta_1,\ldots ,\Delta_l\}$ of $\Delta$. Let $y= <x,v_i>+r_i$ be the equation of the hyperplane $Q_i\subset \mathbb{R}^n\times\mathbb{R}$ containing  points of coordinates $(\alpha ,\nu (\alpha ))$ with $\alpha \in \Verte (\Delta_i)$.

\noindent There is a duality between the subdivision $\tau$ and the subdivision of $\mathbb{R}^n$ induced by $\Gamma_g$, 
where each connected component of
 $\mathbb{R}^n\setminus \Gamma_g$ is dual to some vertex of $\tau_f$ and each $k$-cell of $\Gamma_g$ is dual to some
 $(n-k)$-cell of $\tau$. In particular, each $(n-1)$-cell of $\Gamma_g$ is dual to some edge of $\tau$. 
 If $x\in E_{\alpha\beta}^*\subset \Gamma_g$,  then $<\alpha , x> -\nu (\alpha ) = <\beta , x> -\nu (\beta )$, 
so $<\alpha  -\beta , x - v_i>  = 0$. This means that $v_i$ is a vertex of $\Gamma_g$ dual to some $\Delta_i$ having 
$E_{\alpha\beta}$ as edge.

\begin{definition}
A tropical hypersurface $\Gamma\subset \mathbb{R}^n$ is smooth if and only if its dual subdivision is a triangulation where the Euclidean volume of every triangle is equal to $\frac{1}{n!}$.
\end{definition}

%
%
%


\noindent 
Let  $\mathring{V}\subset (\mathbb{C}^*)^n$ be an algebraic hypersurface defined by a polynomial
${f (z) = \sum_{\alpha_i\in A}a_{\alpha_i}z^{\alpha_i}}$, with support $A = \{ \alpha_1,\ldots , \alpha_l,\alpha_{l+1},\ldots ,\alpha_r\} \subset \mathbb{Z}^n$, and $A' = \{ \alpha_{l+1},\ldots ,\alpha_r\} = \Iim (\ord )$ where $\ord$ is the order mapping from the set of complement components of the amoeba $\mathscr{A}$ of $\mathring{V}$ to $\Delta\cap\mathbb{Z}^n$ (see \cite{FPT-00}).
It was shown by Mikael  Passare and Hans Rullg\aa  (see  \cite{PR1-04}) that the spine $\Gamma$ of the amoeba
$\mathscr{A}$ is  a non-Archimedean amoeba defined by the
tropical polynomial
$$
f_{trop}(x) = \max_{\alpha\in A'}\{ c_{\alpha}+<\alpha , x>\} ,
$$
where $c_{\alpha}$  are a constants  defined by:
\begin{equation}\label{(2)}
c_{\alpha} = \R \left(   \frac{1}{(2\pi i)^n}\int_{\Log^{-1}(x)}\log
  \left|\frac{f(z)}{z^{\alpha}}\right| \frac{dz_1\wedge \ldots \wedge
    dz_n}{z_1\ldots z_n}\right)
\end{equation}
where  $x\in E_{\alpha}$,\, $z = (z_1,\cdots ,z_n)\in (\mathbb{C}^*)^n$. In other words, the spine of
$\mathscr{A}$ is defined as the set of points in $\mathbb{R}^n$
where the piecewise affine linear function $f_{trop}$ is not
differentiable.  
Let us denote by $\tau$ the convex subdivision of $\Delta$ dual
to the tropical variety $\Gamma$. Then the set of vertices of $\tau$ is precisely the image of the order mapping (i.e., $A'$).
By duality, this means that the convex subdivision $\tau =\cup _{i=l+1}^r \Delta_{v_i}$ of $\Delta$ is determined by a piecewise affine linear map $\nu : \Delta\longrightarrow \mathbb{R}$ so that:

\begin{itemize}
\item[(i)]\, $\nu_{\mid \Delta_{v_i}}$ is affine linear for each $v_i$,
\item[(ii)]\, if $\nu_{\mid U}$ is affine linear for some open set $U\subset
  \Delta$, then there  exists $v_i$ such that $U\subset \Delta_{v_i}$.
\item[(iii)]\, $\nu (\alpha ) = - c_\alpha$ for any $\alpha \in \Ima (\ord )$.
\end{itemize}

\noindent We define the {\em generalized  $s$-Passare-Rullg\aa rd function} by the following:

\begin{definition}
Let $s = (s_1,\ldots ,s_l)\in \mathbb{R}_+^l$ and  $\nu_{PR}^s : A
\longrightarrow \mathbb{R}$   be
the function, called the generalized $s$-Passare-Rullg\aa rd function, is defined by:
\[
\nu_{PR}^s(\alpha ) = \left\{ \begin{array}{ll}
-c_{\alpha}& \mbox{if\, $\alpha\in \Iim (\ord )$}\\
 <\alpha_j , a_{v}>+b_{v}+s_j &\mbox{if\, $\alpha = \alpha_j$ \,for\, $j=1,\ldots ,l$},                    
\end{array}
\right.
\]
where $\alpha_j\in\Delta_{v}$,\,  $\Delta_{v}\in \tau$ and $y = <x , a_{v}>+b_{v}$
is the equation of the hyperplane in $\mathbb{R}^n\times\mathbb{R}$ containing the
points of coordinates $(\beta ; -c_{\beta})$ with $\beta\in \Verte (\Delta_{v})$.
\end{definition}

\noindent Assume that we have a hypersurface $\mathring{V}\subset (\mathbb{C}^*)^n$ defined by the polynomial ${f(z)=\sum_{\alpha\in A}a_\alpha z^{\alpha}}$ with $a_\alpha\in \mathbb{C}^*$, $A$ a finite subset of $\mathbb{Z}^n$ and $z^{\alpha} = z_1^{\alpha_1}z_2^{\alpha_2}\ldots z_n^{\alpha_n}$. We denote by $\Delta$ the convex hull of $A$ in $\mathbb{R}^n$ which is the Newton polytope of $f$. We can consider  the family of
 hypersurfaces $\mathring{V}_{f_{(t;\,  s)}}\subset (\mathbb{C}^*)^n $ defined by the following family of polynomials :
\begin{equation}\label{(3)}
f_{(t;\,  s)}(z) =
\sum_{\alpha\in A}\xi_\alpha t^{\nu^s_{PR}(\alpha )} z^{\alpha},
\end{equation}
with $\xi_\alpha = a_\alpha e^{\nu^s_{PR}(\alpha )}$, and we view this family as a deformation of $f$.

Let us denote by  $\mathscr{C}oh_A(\Delta )$  the set of coherent (i.e. convex) triangulations of $\Delta$ such that the set of vertices of all its elements is  contained in  $A$.  For each $\tau\in \mathscr{C}oh_A(\Delta )$, assume $\nu : \Delta \rightarrow \mathbb{R}$  is a convex function defining $\tau$.
 Let  $f^{(\tau )}$ be the non-Archimedean polynomial defined by:
 $$
 f^{(\tau )}(z) = \sum_{\alpha\in A} a_{\alpha}t^{\nu (\alpha )}z^{\alpha}.
 $$
We denote by $co\mathscr{A}_{\mathbb{C}}(f)$ (resp. $co\mathscr{A}_{\mathbb{K}}(f)$) the  complex coamoeba (resp. non-Archimedean coamoeba) of the hypersurface with defining polynomial $f$.

\section{Phase tropical hypersurfaces}

%
%

%
%
%
%
%
%
%
%


\subsection{Phase tropical hypersurfaces} ${}$

\vspace{0.2cm}

For every  strictly positive real  number $t$ we define the self diffeomorphism $H_t$ of $(\mathbb{C}^*)^n$ by :
\[
\begin{array}{ccccl}
H_t&:&(\mathbb{C}^*)^n&\longrightarrow&(\mathbb{C}^*)^n\\
&&(z_1,\ldots ,z_n)&\longmapsto&\left(\mid z_1\mid^{-\frac{1}{\log t}}\dfrac{z_1}{\mid
  z_1\mid},\ldots ,\mid z_n\mid^{-\frac{1}{\log t}}\dfrac{z_n}{\mid z_n\mid} \right).
\end{array}
\]
This defines a new  complex structure on $(\mathbb{C}^*)^n$
denoted by $J_t = (dH_t)^{-1}\circ J\circ (dH_t)$ where $J$ is the
standard complex structure.

\noindent A $J_t$-holomorphic hypersurface $\mathring{V}_t$ is a 
holomorphic hypersurface with respect to the $J_t$ complex structure on
$(\mathbb{C}^*)^n$. It is equivalent to say that $\mathring{V}_t = H_t(\mathring{V})$ where
$\mathring{V}\subset (\mathbb{C}^*)^n$ is an holomorphic hypersurface for the
standard complex structure $J$ on $(\mathbb{C}^*)^n$.

Recall that the Hausdorff distance between two closed subsets $A,
B$ of a metric space $(E, d)$ is defined by:
$$
d_H(A,B) = \max \{  \sup_{a\in A}d(a,B),\sup_{b\in B}d(A,b)\}.
$$
Here $E =\mathbb{R}^n\times (S^1)^n$ is equipped  with the distance
defined as the product of the
Euclidean metric on $\mathbb{R}^n$ and the flat metric on $(S^1)^n$.

\begin{definition} A phase tropical hypersurface $\mathring{V}_{\infty}\subset
  (\mathbb{C}^*)^n$ is the limit (with respect to the Hausdorff
  metric on compact sets in $(\mathbb{C}^*)^n$) of a  sequence of a
  $J_t$-holomorphic hypersurfaces $\mathring{V}_t\subset (\mathbb{C}^*)^n$ when
  $t$ tends to $\infty$.
\end{definition}

We have an algebraic definition of phase tropical hypersurfaces  in case of curves (called complex tropical curves)(see \cite{M2-04}) as follows :

\noindent Let $ a \in \mathbb{K}^*$ be the Puiseux series
${a = \sum_{j\in A_a}\xi_jt^j}$ with $\xi\in \mathbb{C}^*$
and $A_a\subset \mathbb{R}$ is a well-ordered set with smallest element
Then we have a non-Archimedean valuation on $\mathbb{K}$ defined by
$\val (a ) = - \min A_a$. We complexify the valuation
map as follows :
\[
\begin{array}{ccccl}
w&:&\mathbb{K}^*&\longrightarrow&\mathbb{C}^*\\
&&a&\longmapsto&w(a ) = e^{\val (a )+i\arg (\xi_{-\val
    (a )})}.
\end{array}
\]
Let  $\Arg$ be the argument map $\mathbb{K}^*\rightarrow S^1$ defined by: for any  Puiseux series  ${a = \sum_{j\in A_a}\xi_jt^j}$, we set  $\Arg (a) = e^{i\arg (\xi_{-\val (a)})}$ (this map extends the map $\mathbb{C}^*\rightarrow S^1$ defined by $\rho e^{i\theta} \mapsto e^{i\theta}$ which we denote by $\Arg$).

\vspace{0.3cm}

\noindent Applying this map coordinatewise we obtain a map :
\[
\begin{array}{ccccl}
W:&(\mathbb{K}^*)^n&\longrightarrow&(\mathbb{C}^*)^n
\end{array}
\]

\begin{theorem}[Mikhalkin, 2002]  The set  $\mathring{V}_{\infty}\subset (\mathbb{C}^*)^n$
  is a phase tropical hypersurface if and only if there
  exists an algebraic hypersurface
  $\mathring{V}_{\mathbb{K}}\subset(\mathbb{K}^*)^n$ over $\mathbb{K}$ such that
  $\overline{W(\mathring{V}_{\mathbb{K}})} = \mathring{V}_{\infty}$, where $\overline{W(\mathring{V}_{\mathbb{K}})}$ is the closure of $W(\mathring{V}_{\mathbb{K}})$ in  $(\mathbb{C}^*)^n \approx \mathbb{R}^n\times (S^1)^n$ as  a Riemannian manifold with metric defined by the standard Euclidean metric of $\mathbb{R}^n$ and  the standard  flat metric of the real  torus.
\end{theorem}

Let $f_t(x)=\sum_{\alpha}a_\alpha t^{-v(\alpha)}z^j$ be a polynomial with a parameter $t$, and $\mathring{V}_t=\{f_t=0\}\subset (\mathbb C^*)^n$. The family of $f_t$ can be viewed as a single polynomial in $\mathbb K[z_1^{\pm 1},\cdots,z_n^{\pm 1}]$.  We have the following theorems (see  \cite{M2-04}, \cite{M3-04}, and \cite{R1-01}):

\begin{theorem}[Mikhalkin, Rullg{\aa}rd  (2001)]
The amoebas $\mathscr{A}_t$  of $\mathring{V}_t$ converge in the Hausdorff metric to the non-archimedean amoeba $\mathscr{A}_{\mathbb K}$ when $t\to \infty$.
\end{theorem} 

\begin{theorem}[Mikhalkin]
The sets $H_t(\mathring{V}_t)$ converge in the Hausdorff metric to $W(\mathring{V}_{\mathbb K})$ when $t\to\infty$.
\end{theorem}

\subsection{Tropical localization} ${}$

\vspace{0.2cm}
\noindent  Let $\nu$ be the piecewise affine linear map  defined in Section 2, and 
 $\tilde{\Delta}$ be the extended polyhedron of $\Delta$ associated
to $\nu$, that is the convex hull of the set $\{  (\alpha ,u)\in
\Delta\times \mathbb{R} \,|\, u\geq \nu (\alpha ) \}$. For any
$\Delta_{v_i}\in \tau$, let $\lambda (x) = <x, a_{v_i}> + b_{v_i}$ be the affine linear
map defined on $\Delta$ such that $\lambda_{\mid \Delta_{v_i}} = \nu_{\mid
  \Delta_{v_i}}$ where $<~,~>$ is the scalar product in $\mathbb{R}^n$, $a_{v_i}=(a_{v_i,\, 1},\ldots ,a_{v_i,\, n})\in\mathbb{R}^n$ (which is the coordinates of the vertex of the spine  $\Gamma$, dual to $\Delta_{v_i}$), and $b_{v_i}$ is a real number. Let $s\in \mathbb{R}_+^l$ as above and put $\nu ' = \nu_{PR}^{(s)} -\lambda$ and we define the family of polynomials
 $\{ {f'}_{(t;\, s)} \}_{t\in (0,\, \frac{1}{e}]}$ by:
$$
f_{(t,s)} '(z) =\sum_{\alpha\in A}\xi_{\alpha}t^{\nu '(\alpha )}z^{\alpha},
$$
where $\xi_{\alpha}\in \mathbb{C}$. Then we have:
\begin{align*}
&f_{(t,s)} '(z) =t^{-b_v}\sum_{\alpha\in A}\xi_{\alpha}t^{\nu_{PR}^{(s)} (\alpha
  )}(z_1t^{-a_{v_i,\, 1}})^{\alpha_1}\ldots (z_nt^{-a_{v_i,\,
  n}})^{\alpha_n}\nonumber \\
&\phantom{f_{(t,s)} '(z)}=t^{-b_v}f_{(t;\, s)}\circ \Phi^{-1}_{\Delta_{v_i},\, t}(z),\nonumber
\end{align*}
 where $f_{(t;\, s)}$  is the polynomial defined in \eqref{(3)}, and $\Phi_{\Delta_{v_i},\, t}$ is the self diffeomorphism of
 $(\mathbb{C}^*)^n$ defined by:
\[
\begin{array}{ccccl}
\Phi_{\Delta_{v_i},\, t}&:&(\mathbb{C}^*)^n&\longrightarrow&(\mathbb{C}^*)^n\\
&&(z_1,\ldots ,z_n)&\longmapsto&(z_1t^{a_{v_i,\, 1}},\ldots ,z_nt^{a_{v_i,\,
  n}} ).
\end{array}
\]
This means that the polynomials ${f'}_{(t;\, s)}$ and $f_{(t;\, s)}\circ
\Phi^{-1}_{\Delta_{v_i},\, t}$ define the same hypersurface. So we have:
$$
\mathring{V}_{{f'}_{(t;\, s)}} = \mathring{V}_{f_{(t;\, s)}\circ
\Phi^{-1}_{\Delta_{v_i},\, t}} = \Phi_{\Delta_{v_i},\, t}(\mathring{V}_{f_{(t;\, s)}}),
$$
where  $\mathring{V}_g$ denotes  algebraic hypersurface in $(\mathbb{C}^*)^n$ with defining polynomial $g$.
Let $U_{v_i}$ be a small ball in $\mathbb{R}^n$ with center the
vertex of $\Gamma_{(t;\, s)}$ dual to $\Delta_{v_i}$ where $\Gamma_{(t;\, s)}$ is the
spine of the amoeba $\mathscr{A}_{H_t(\mathring{V}_{f_{(t;\, s)}})}$ where $H_t$
denotes the self diffeomorphism of $(\mathbb{C}^*)^n$ defined as in Subsection 3.1,
  and $\Log_t = \Log \circ H_t$.
Let $f_{(t;\, s)}^{\Delta_{v_i}}$ be the truncation of $f_{(t;\, s)}$ to $\Delta_{v_i}$, and 
$\mathring{V}_{\infty ,\,\Delta_{v_i}}$ be the complex tropical hypersurface with tropical coefficients of index $\alpha\in\Delta_{v_i}$ (i.e., $\mathring{V}_{\infty ,\,\Delta_{v_i}} = \lim_{t\rightarrow 0} H_t(\mathring{V}_{f_{(t;\, s)}^{\Delta_{v_i}}})$). Using Kapranov's theorem (see \cite{K-00}), we obtain the following Proposition (called a tropical localization by Mikhalkin, see \cite{M2-04}):

\begin{proposition} Let $s$ be in $\mathbb{R}_+^l$.  For any $\varepsilon >0$ there exists $t_0$ such that if $t\geq t_0$ then the image under $\Phi_{\Delta_{v_i},\, t}\circ H_t^{-1}$ of $H_t(\mathring{V}_{f_{(t;\, s)}})\cap \Log^{-1}(U_{v_i})$ is contained in the $\varepsilon$-neighborhood of the image under $\Phi_{\Delta_{v_i},\, t}\circ H_t^{-1}$ of the phase tropical hypersurface $\mathring{V}_{\infty ,\,\Delta_{v_i}}$ corresponding to the family $\{\mathring{V}_{f_{(t;\, s)}}\}_{t}$, with respect to the product metric in $(\mathbb{C}^*)^n\approx\mathbb{R}^n\times (S^1)^n$.
\end{proposition}

\begin{proof}
By decomposition of $f_{(t,s)} '$, we obtain:
\begin{equation}\label{(4)}
f_{(t,s)} '(z) =t^{-b_v}\sum_{\alpha\in \Delta_v\cap A}\xi_{\alpha}t^{\nu
  (\alpha ) -<\alpha ,a_v>}z^{\alpha} \,\,  +\,\, \sum_{\alpha\in
  A\setminus \Delta_v}  \xi_{\alpha}t^{\nu
  (\alpha ) -<\alpha ,a_v>-b_v}z^{\alpha}. 
\end{equation}
On the other hand,  we have the following commutative diagram:
\begin{equation}\label{(5)}
\xymatrix{
(\mathbb{C}^*)^n\ar[rr]^{\Phi_{\Delta_v,t}}\ar[d]_{\Log_t}&&
(\mathbb{C}^*)^n\ar[d]^{\Log_t}\cr
\mathbb{R}^n\ar[rr]^{\phi_{\Delta_v}}&&\mathbb{R}^n,
}
\end{equation}
such that if $v=(a_{v,\, 1},\ldots ,a_{v,\, n})\in \mathbb{R}^n$ is
the vertex of the tropical hypersurface $\Gamma$ dual to the
element $\Delta_v$ of the subdivision $\tau$, then
$\phi_{\Delta_v}(x_1,\ldots , x_n)=(x_1-a_{v,\, 1},\ldots
,x_n-a_{v,\, n})$. Let $U_v$  be a small open ball in
$\mathbb{R}^n$ centered at $v$.

\noindent Assume that $\Log_t(z)\in \phi_{\Delta_v}(U_v)$ and $z$ is not
singular in $\mathring{V}_{t}$. Then the second sum in \eqref{(4)} converges to zero
when $t$ goes  to infinity, because by the choice of $z$ and $U_v$, the
tropical monomials in $f_{trop,\, (t,s)}'$, corresponding to lattice points
of $\Delta_v$,  dominates the monomials corresponding to lattice points
of $A\setminus \Delta_v$. But the first sum in \eqref{(4)} is just a
polynomial defining the hypersurface $ \Phi_{\Delta_v,\, t} (V_{f_{(t,s)}^{\Delta_v}})$. 

\noindent By the commutativity of diagram \eqref{(5)}, if
$z\in \mathring{V}_{f_t'}$ is such that $\Log_t(z)\in \phi_{\Delta_v}(U_v)$
then $\Log_t\circ \Phi_{\Delta_v,t}^{-1}(z)\in U_v$, and hence
$H_t( \Phi_{\Delta_v,t}^{-1}(z))\in \Log^{-1}(U_v)$. So, the image under 
$\Phi_{\Delta_v,\, t}\circ H_t^{-1}$ of
$H_t(\mathring{V}_{f_{(t,s)}})\cap \Log^{-1}(U_v)$ is contained in an
$\varepsilon$-neighborhood of  the image under  $\Phi_{\Delta_v,\, t}\circ H_t^{-1}$ of
  $H_t(\mathring{V}_{f_{(t,s)}^{\Delta_v}})$ for
sufficiently large $t$ and the proposition is proved because
$\mathring{V}_{\infty ,\,\Delta_v}$ is the limit when $t$ tends to $\infty$ of
the sequence of $J_t$-holomorphic hypersurfaces
$H_t(V_{f_{(t,s)}^{\Delta_v}})$ (by taking a discrete sequence  $t_k$
converging to $\infty$ if necessary). In particular the set of arguments of $\mathring{V}_{\infty ,\, f}\cap \Log^{-1}(v)$ 
is contained in  the set of arguments of $V_{\infty ,\,\Delta_v}$ i.e.,  $\Arg (\mathring{V}_{\infty ,\, f}\cap \Log^{-1}(v))\subseteq \Arg (\mathring{V}_{\infty ,\,\Delta_v})$. If it is not the case, we can get away too after applying $\Phi_{\Delta_v,\, t}\circ H_t^{-1}$ for sufficiently large $t$.
 
\end{proof}

\subsection{Toric varieties} ${}$
\vspace{0.2cm}

\noindent To every convex polyhedron $\Delta\subset \mathbb R^n$ with integer vertices, there is a complex toric variety ${X}_{\Delta}$ containing $(\mathbb C^*)^n$. Indeed, we can consider the Veronese embedding $\rho: (\mathbb C^*)^n \rightarrow \mathbb{CP}^{\# (\Delta\cap \mathbb Z^n)-1}$ defined by the monomial map associated to $\Delta\cap \mathbb Z^n$: $(z_1,\cdots,z_n)\mapsto z_1^{\alpha_1}z_2^{\alpha_2}\cdots z_n^{\alpha_n}$, for each $\alpha:=(\alpha_1,\cdots,\alpha_n)\in \Delta\cap \mathbb Z^n$; and   $X_{\Delta}$ is defined as the closure of the image of $(\mathbb C^*)^n$. Then  the Fubini-Study symplectic form on the projective spaces $\mathbb{CP}^{\# (\Delta\cap \mathbb Z^n)-1}$ defines a natural symplectic form on $X_{\Delta}$. In particular we obtain a symplectic form $\omega_{\Delta}$ on $(\mathbb C^*)^n$ invariant under the Hamiltonian action of the real torus $(S^1)^n$. This gives a moment map $\mu_{\Delta}$ with respect to $\omega_{\Delta}$:
\[
\begin{matrix}
\mu_{\Delta}&:&(\mathbb C^*)^n&\longrightarrow& \Delta\\
            & &z &\mapsto &\dfrac{\displaystyle\sum_{\alpha\in \Delta\cap \mathbb Z^n} \sum_{i=1}^n \alpha_i|z_i^{2\alpha_i}|}{\displaystyle\sum_{\alpha\in \Delta\cap \mathbb Z^n}\sum_{i=1}^n|z_i^{2\alpha_i}|},
\end{matrix}
\]
which  is an embedding with image the interior of $\Delta$. 
\begin{equation}
\xymatrix{
(\mathbb{C}^*)^n\ar[rr]^{\Log}\ar[dr]_{\mu_{\Delta}}&&\mathbb{R}^n\ar[dl]^{\Psi_\Delta}\cr
&\Delta.
}\nonumber
\end{equation}
The maps $\Log$ and $\mu_{\Delta}$ both have orbits $(S^1)^n$ as fibers, and we obtain a reparametrization of $\mathbb{R}^n$  which we denote by  $\Psi_{\Delta}$ (see \cite{GKZ-94}).

\begin{definition}
Let $\Gamma\subset \mathbb R^n$ be an $n$-dimensional balanced polyhedral complex, and $\Delta$ its dual convex lattice polyhedron. $\overline{\Gamma}\subset \Delta$ is the compactification of $\Gamma$ by taking $\Psi_{\Delta}(\Gamma)$ in $\Delta$. $\overline{\Gamma}\backslash \Psi_{\Delta}(\Gamma)$ is called the boundary of $\overline{\Gamma}$. 
\end{definition}

Let $f$ be a Laurent polynomial in $\mathbb C[z_1^{\pm 1},\cdots , z_n^{\pm 1}]$, and $\Delta$ be its Newton polytope. Let $\mathring{V} :=\{z\in (\mathbb C^*)^n\,|\,f(z)=0\}$ be the hypersurface in $(\mathbb C^*)^n$ with defining polynomial $f$. Let $X_{\Delta}$ be the complex toric variety as defined before. We denote by $V$ the closure of the hypersurface $\mathring{V}$ in $X_{\Delta}$.

\vspace{0.2cm}

\noindent Let $\Delta$ be a compact convex lattice polyhedron such that the singularity of its corresponding toric variety $X_{\Delta}$ are on the vertices of $\Delta$. Let $(\mathbb C^*)^{\# (\Delta\cap \mathbb Z^n)}$ be the set of all polynomial $f(z)=\sum_{\alpha\in \Delta\cap\mathbb Z^n}a_{\alpha}z^{\alpha}$ such that $a_{\alpha}\ne 0$. Then for a generic choice of a polynomial,  the closure $V$ in $X_{\Delta}$ of the zero set of $f$ is a smooth hypersurface transverse to all toric subvarieties $X_{\Delta'}$, corresponding to the faces $\Delta'\subset \Delta$. In particular, all such hypersurfaces $V$ are diffeomorphic, even symplectomorphic if they are equipped with the symplectic form coming from the one of $X_{\Delta}$.

%
%
%
%
%
%
%

\section{Examples of coamoebas and phase tropical hypersurfaces}

\begin{itemize}
\item[(a)] Let $\mathring{V}$ be the line in $(\mathbb{C}^*)^2$ defined by the polynomial 
$f(z,w)=r_1e^{i\alpha_1}z+r_2e^{i\alpha_2}w+r_3e^{i\alpha_3}$ where $r_i$ are real positive numbers and $\alpha_1>\alpha_3>\alpha_2>0$. Then its coamoeba is as displayed in Figure 1. The equations of the  external hyperplanes are given by $(1)$ \, $y = x + \alpha_1 - \alpha_2 + (2k+1)\pi$, \, $(2)$\, $x = \alpha_3 - \alpha_1 + (2l+1)\pi$, and $(3)$\, $y =  \alpha_3 - \alpha_2 + (2m+1)\pi$ with $k,\, l$ and $m$ in $\mathbb{Z}$ (the external hyperplanes are seen in $\mathbb{R}^2$ the universal covering of the torus).

\begin{figure}[h!]
\begin{center}
\includegraphics[angle=0,width=0.5\textwidth]{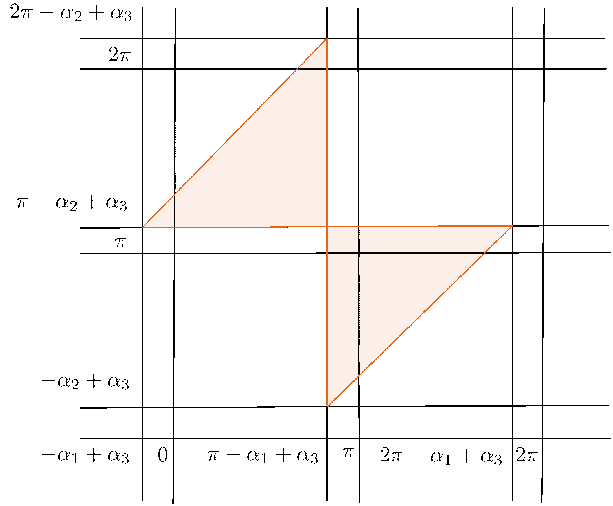}\quad
\caption{The coamoeba of the line in $(\mathbb{C}^*)^2$ defined by the polynomial $f(z,w)=r_1e^{i\alpha_1}z+r_2e^{i\alpha_2}w+r_3e^{i\alpha_3}$ where $r_i$ are real positive numbers and $\alpha_1>\alpha_3>\alpha_2>0$.}
\label{}
\end{center}
\end{figure}

We can remark that in this case there are no extra-pieces, and all the boundary of the closure of this coamoeba in the torus is contained in three external hyperplanes.
\item[(b)] Consider   now the example of a parabola.
Let $\mathring{V}_f\subset (\mathbb{C}^*)^2$ the curve defined by the polynomial $f(z,w) = w-z^2+2z-\lambda$ with $\lambda >1$. Consider the parametrization defined by :
$$
\left\lbrace
\begin{array}{l}
z(r, \alpha ) = r{e}^{i\alpha},\\
w(r, \alpha ) = r^2{e}^{2i\alpha}-2r{e}^{i\alpha}+\lambda ,
\end{array}
\right.
$$
with $r>0$ and $\alpha\in [0,2\pi ]$. We have to compute the argument
of $r^2{e}^{2i\alpha}-2r{e}^{i\alpha}+\lambda$, with $r\in
\mathbb{R}_+^*$. Let $a=\lambda -1$, so we have $w(r,\alpha ) =
(r{e}^{i\alpha} - 1)^2 + a$ and then $\beta = \arg (w(r,\alpha )) =
\arg \left[   \dfrac{(r{e}^{i\alpha} - 1)- i\sqrt{a}}{(r{e}^{-i\alpha}
    - 1)- i\sqrt{a}}\right]$.
\begin{itemize}
\item[(i)] Let $0\leq\alpha\leq \arctan\sqrt{a}$ then
  $0\leq\beta\leq 2\alpha$ if $1+\tan^2\alpha\leq r^2<\infty$ and
  $g_\alpha (r)\leq \beta\leq 2\pi$ if $0<r^2<1+\tan^ 2\alpha$ where
  for each $\alpha$,\,  $g_\alpha$
  is a differentiable function with one maximum in the interval
  $0<r^2<1+\tan^2\alpha$ (see  Figure 2);
\item[(ii)] If $\pi\geq\alpha\geq\arctan\sqrt{a}$ then
  $2\alpha\leq\beta\leq 2\pi$;
\item[(iii)] For $\alpha > \pi$ we have the conjugate of the sets
  in (i) and (ii).
\end{itemize}

\begin{figure}[h!]
\begin{center}
\includegraphics[angle=0,width=0.4\textwidth]{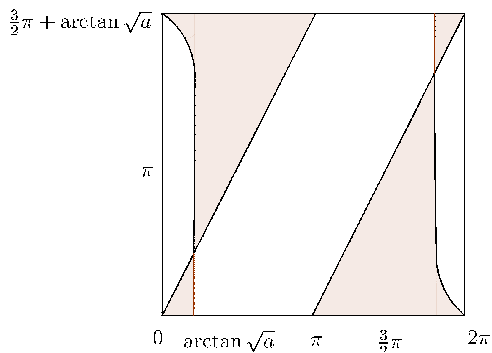}\quad
\caption{Coamoeba of a parabola.}
\label{}
\end{center}
\end{figure}

We can view a parabola as   an algebraic curve 
$\mathring{V}_{f_{\mathbb{K}}}$ over the field of the Puiseux series with
real powers $\mathbb{K}$, defined by the polynomial
$f_{\mathbb{K}}(z,w)=f_t(z,w) =t^0w-t^0z^2+2t^0z-t^{-\Log \lambda}$ with $z,\, w\in \mathbb{K}^*$
and $t\in \mathbb{R}_+^*$.
 It is clear that the limit of the coamoebas of the
curves $\mathring{V}_{f_t}$ converge to the coamoeba of the phase tropical
curve with tropical coefficients
$a_{01}=1,\, a_{00}=-\lambda$ and
$a_{20}=-1$, which are the coefficients with index in $\Verte (\tau )$
where $\tau$ is the triangulation of the Newton polygon of $f$ dual to
$\Gamma$, with $\Gamma$ the tropical curve that is the spine of the amoeba of
$\mathring{V}_f$ (see Figure 3, the coamoeba of a phase tropical parabola).

\begin{figure}[h!]
\begin{center}
\includegraphics[angle=0,width=0.3\textwidth]{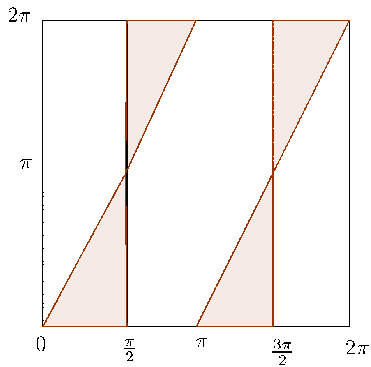}\quad
\caption{Coamoeba of a  parabola with coefficients only in the vertices of the Newton polygon of its defining polynomial.}
\label{}
\end{center}
\end{figure}
We can see in Figure 2 extra-pieces in the coamoeba of our parabola.

\item[(c)]
Let $V_\lambda$ be the complex curve defined by the polynomial
$f(z,w) = \lambda + z + w + zw$ with $\lambda \in \mathbb{R}^*$ with Newton
polygon the standard square of vertices $(0,0), (1,0), (0,1)$ and
$(1,1)$.
\end{itemize}
\begin{itemize}
\item[$1^{st}$ {\it case}.]
 Assume  $0<\lambda < 1$, and we  parametrize
$z=re^{i\alpha}$ with $\alpha \in [0, 2\pi ]$ and $r\in
\mathbb{R}^*_+$. So $\arg (w(r,\alpha )) = \theta (r,\alpha )$ with:
$$
\theta (r,\alpha ) = \arcsin \left( \frac{-r (1-\lambda )\sin
  \alpha}{((\lambda + r(1+\lambda )\cos \alpha +r^2)^2+ r^2(1-\lambda
  )^2\sin^2\alpha )^{\frac{1}{2}}}\right)
$$
and we have $\dfrac{\partial \theta}{\partial r}(r, \alpha ) = 0$ if
and only if $r=\pm \sqrt{\lambda}$, so
$r=\sqrt{\lambda}$ and the maximum  of the argument of $w$ is
attained at $r=\sqrt{\lambda}$, this means that we have
$$
\theta_{\max}(\alpha )=  \arcsin \left( \frac{-\sqrt{\lambda} (1-\lambda )\sin
  \alpha}{((2\lambda +\sqrt{\lambda} (1+\lambda )\cos \alpha
  )^2+ \lambda (1-\lambda
  )^2\sin^2\alpha )^{\frac{1}{2}}}\right)
$$
If $0<\lambda < 1$ it can be viewed as a parameter, and hence as an
element of $\mathbb{K}^*$, which means that the curve $V_\lambda$ is viewed as
an algebraic curve over $\mathbb{K}$, i.e. $V_{\lambda}^{\mathbb{K}} =\{ (z,w)\in (\mathbb{K}^*)^2
\,|\,\lambda + z + w + zw =0 \}$ and $\Log_{\mathbb{K}}(V_{\lambda}^{\mathbb{K}})$ is 
the tropical curve with tropical polynomial
$f_{trop}(x,y) = \max \{ x,y,x+y, -1 \}$. We have $\Log^{-1}(v_1)\cap
W(V_{\lambda}^{\mathbb{K}})$ is the union of the two sets of $S^1\times S^1$ with
boundary the two half of the cycles $\delta_1 =\{ \alpha =\pi \}$ and
$\delta_2 =\{ \beta =\pi \}$ and the half of the cycle defined by the
graph of the function $\theta_{\max}$, which is homotopic to the
product of $\delta_1$ and $\delta_2$. We have the same
result for the vertex $v_2$.

\begin{figure}[h!]
\begin{center}
\includegraphics[angle=0,width=0.6\textwidth]{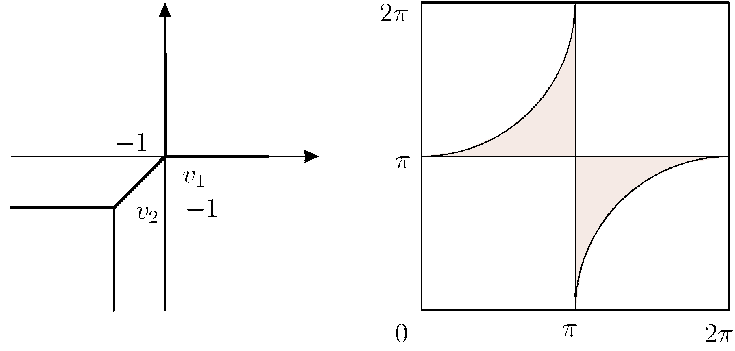}\quad
\caption{The spine of the amoeba of the hyperbola defined by the polynomial $f_{\lambda}$ with $0<\lambda<1$ and its coamoeba.}
\label{}
\end{center}
\end{figure}

\vspace{0.3cm}

\item[$2^{nd}$ {\it case}.]
 Suppose $\lambda > 1$,  and  let $\tau
=\frac{1}{\lambda}$. So, $\lambda = \tau^{-1}$ and then $V_{\tau}^{\mathbb{K}} =\{ (z,w)\in ({\mathbb{K}}^*)^2
\,|\, \tau^{-1} + z + w + zw =0 \}$. Hence $\Log_{\mathbb{K}}(V_{\tau}^{\mathbb{K}})$ is 
the tropical curve  with tropical polynomial
$f_{trop}(x,y) = \max \{ x,y,x+y, +1 \}$. Hence,  we have:
\begin{align*}
\Log^{-1}(v_1)\cap
&W(V_{\tau}^{\mathbb{K}})=\{ (\alpha ,\beta) \in S^1\times S^1 / 0\leq \alpha\leq
\pi ,\,\, \theta_{\max}(\alpha )\leq\beta\leq \pi \}\\
&\phantom{W(V_{\tau}^{\mathbb{K}})=}\cup \{ (\alpha ,\beta) \in S^1\times S^1 / \pi\leq \alpha\leq
2\pi ,\,\, \pi\leq\beta\leq   \theta_{\max}(\alpha )\}
\end{align*}

\vspace{0.3cm}

\item[$3^{rd}$ {\it case}.] Assume $\lambda =1=t^0$, so we have $f_1(z,w) =
(1+z)(1+w)$, and the corresponding tropical curve is the union of 
two axes, and  $\Log^{-1}(v_1)\cap
W(V_1^{\mathbb{K}})$ is the union of  two circles (the valuation of the
constant coefficient is zero in this case).

\vspace{0.3cm}

\item[$4^{th}$ {\it case}.] Suppose $\lambda < 0$ and $\lambda \ne -1$. If
$\mid \lambda \mid <1$, then consider  $\lambda$ as a parameter and we
have the tropical curve of the first case (it means that the
valuation of the constant coefficient is negative). So, if we put $z(t)
=t^{-x}e^{i\alpha}$ then $w(t)_{\alpha} = -\left(
\dfrac{t-t^{-x}e^{i\alpha}}{-1-t^{-x}e^{i\alpha}} \right)$ and then $\Log^{-1}(v_1)\cap
W(V_{\lambda}^{\mathbb{K}})$ is the closure in $S^1\times S^1$ of the set
$$\left(\alpha , \lim_{t\rightarrow 0,\, x\rightarrow 0} \arg (w(t)_{\alpha}\right)$$
with $0\leq \alpha \leq 2\pi$. We then obtain  the union  of two
triangles. For the second vertex  we have $\Log^{-1}(v_2)\cap
W(V_{\lambda}^{\mathbb{K}})$ is the closure in $S^1\times S^1$ of the set
$$\left(\alpha , \lim_{t\rightarrow 0, x\rightarrow 1} \arg (w(t)_{\alpha}\right)$$
with $0\leq \alpha \leq 2\pi$, and we obtain the union of  two
triangles.

\vspace{0.3cm}

\item[$5^{th}$ {\it case}.] Suppose $\lambda < -1$ and write $\lambda
= \tau^{-1}$ with $-1<\tau <0$. So we have the tropical curve of the
second case (this means that the valuation of the constant
coefficient is positive).

\end{itemize}

\section{A differential  structure on phase tropical hypersurfaces}

\vspace{0.2cm}
\subsection{A differential  structure on phase tropical hyperplanes} ${}$

\vspace{0.2cm}

In \cite{M2-04}, Mikhalkin gives the following definition of a generalized pair-of-pants:
\begin{definition}\label{definition-pants}
Let $\mathscr{H} \subset \mathbb{CP}^n$ be an arrangement of $n+2$ generic hyperplanes in $\mathbb{CP}^n$. Let $\mathscr{U}\subset \mathbb{CP}^n$ be the union of their tubular $\varepsilon$-neighborhood for a small  $0<\varepsilon\ll1$. The complement $\overline{\mathscr{P}}_n=\mathbb{ CP}^n\backslash   \mathscr{U}$ is called the $n$-dimensional pair-of-pants, and $\mathscr{P}_n=\mathbb{CP}^n\backslash \mathscr{H}$ is called the $n$-dimensional open pair-of-pants. 
\end{definition}

As $\mathscr{H}\subset \mathbb{CP}^n$ is unique up to the action of the projective special linear group $PSL_{n+1}(\mathbb C)$, then $\mathscr{P}_n$ can be given a canonical complex structure. The one dimensional pair-of pants 
$\mathscr{P}_1$ is diffeomorphic to the Riemann sphere punctured at 3 points.
  Moreover,  Mikhalkin constructs a foliation $\mathscr F$ of the complement in $\mathbb{R}^n$ of the complex defined by the standard tropical hyperplane $\Gamma_n$. As before, if $v\in \Gamma$ is a vertex, then there exists a neighborhood ${U}_v$ of $v$ in $\Gamma$ and an affine linear transformation $F$ with linear part $A_v$ in $SL_n(\mathbb Z)$ such that up to a translation in $\mathbb R^n$, ${}^tA^{-1}_v({U}_v)$ is a neighborhood of the origin in $\Gamma_n$. Let $W_v$ be a neighborhood of $\overline{F({U}_v)}$. According to Mikhalkin, a partition of unity gives   a foliation $\mathscr{F}_{\Gamma}$ of a neighborhood $W$ of $\Gamma$.

\noindent Let $\pi_{\mathscr{F}_{\Gamma}}:W(\Gamma)\to \Gamma$ the projection along $\mathscr{F}_{\Gamma}$. By Theorem 5.4 of Mikhalkin and Rullgard, $\Log_t(V_t)\subset W(\Gamma)$ for $t\gg 0$. Let $$\lambda_t:=\pi_{\mathscr{F}_{\Gamma}}\circ \Log_t :V_t\to \Gamma .$$

\vspace{0.2cm}

The example of hyperplanes in the projective space is fundamental for our Theorem \ref{A}. So, let $H=\{ (z_1,\ldots,z_n)\in \mathbb{C}^n\,| \, z_1+\cdots + z_n +1=0\}\subset \mathbb{CP}^n$ be a hyperplane.  Consider its toric part $\mathring{H}= H\cap (\mathbb{C}^*)^n$. Let us denote by $\mathscr{A}_n\subset\mathbb{R}^n$ the amoeba of $\mathring{H}$ and by 
$\Gamma_n\subset \mathbb{R}^n$ the tropical hyperplane defined by the tropical polynomial:
$$
f_{trop}(x_1,\ldots,x_n) = \max \{ 0,x_1,\ldots, x_n\}.
$$
It is well known that $\Gamma_n\subset \mathscr{A}_n$  and it is called the spine of the amoeba $\mathscr{A}_n$. Moreover, $\Gamma_n$ is a strong deformation retract of $\mathscr{A}_n$ (see \cite{PR1-04}).  The number of  connected components of the complement of the amoeba $\mathscr{A}_n$ in $\mathbb{R}^n$ is equal to $n+1$. Each component  $\mathscr{C}_i$ of $\mathbb{R}^n\setminus \mathscr{A}_n$ is equal to the subset of $\mathbb{R}^n$  where one the functions $\{ 0,x_1,\ldots, x_n\}$ is maximal.


\noindent Let us recall   Mikhalkin's construction of the foliation mentioned above (\cite{M2-04}, Section 4.3)  to  obtain a  singular foliation of the amoeba $\mathscr{A}_n$. More precisely, let $\mathscr{L}_i$ be the foliation of the complement component of $\Gamma$ corresponding to $x_i$ (i.e., the set of $\mathbb{R}^n$ where the tropical polynomial $f_{trop}$ achieved its maximum) into straight lines parallel to the gradient $v_i:=\frac{\partial}{\partial x_i}$ of $x_i$ for $i=1,\ldots, n$ and in the component corresponding to the constant function equal to $0$ we consider the foliation into straight lines parallel to the vector with coordinates $v_0=(1,\ldots,1)$. Consider $\pi_i: \mathscr{C}_i\rightarrow\Gamma_n$ the linear projection onto $\Gamma_n$ and parallel to the vector $v_i$. Let $\pi$ the following map:
$$
\pi: \mathscr{A}_n\setminus\Gamma_n \rightarrow \Gamma_n,
$$
where $\pi_{|\mathscr{C}_i\cap\mathscr{A}_n} = {\pi_i}_{|\mathscr{A}_n}$ for each $i=0, 1,\ldots n$. The  foliations of the $\mathscr{C}_i$'s glue to a global foliation $\mathscr{L}$ of $\mathscr{A}_n$ which has singularities at $\Gamma_n$ and the leaves passing through a point $p$ in an open $(n-1-k)$-cell of $\Gamma_n$ is diffeomorphic to the union of $k+2$ segments having a common boundary point $p$ (in other word a cone over $k+2$ points). We can smooth the foliation $\mathscr{L}$ over all open $(n-1)$-cells of $\Gamma_n$, but not at the lower dimensional cells because their leaves are not even a topological manifolds. The only leaves diffeomorphic to a manifold are those passing through open $(n-1)$-cells which are diffeomorphic to the  closed interval $[-1,+1]$. Let us denote the foliation obtained by this smoothing by $\mathscr{F}$.

\begin{proposition}
A phase tropical hyperplane ${H}_\infty\subset (\mathbb{C}^*)^n$ is diffeomorphic to a hyperplane in the projective space $\mathbb{CP}^n$ minus $n+1$  generic hyperplanes.
\end{proposition}
\begin{proof}
 Since each phase tropical hyperplane is a translated  in  $(\mathbb{C}^*)^n$ of the following phase tropical hyperplane ${H}_\infty = W(\{(z_1,\ldots ,z_n)\in (\mathbb{K}^*)^n | \, z_1+\cdots+z_n+1=0\})$, then it suffices to consider this case.
Let us start by the case of phase plane  tropical line in $(\mathbb{C}^*)^2$.  In the case of lines the inverse image by the logarithmic map of the vertex of the tropical line $\Gamma := \Log (\mathscr{H})$ is a union of two triangles whose vertices pairwise identified, and the inverse image by the logarithmic map of any point in the interior of its rays is a circle  (see Example (a)). This means that the inverse image of each ray is  a holomorphic annulus $\mathscr{R}_j$ for $j=1,2,3$. It is clear now that a phase tropical line in $(\mathbb{C}^*)^2$ is diffeomorphic to a sphere punctured in three points. In fact, if we denote $v_0$ the vertex of  $\Gamma$  and $\mathcal{R}_j$ for $j=1, 2, 3$ are the three rays going to the infinity, then the phase   tropical line in $(\mathbb{C}^*)^2$ is diffeomorphic the the gluing of 
the closure $\overline{\Log^{-1}(v_0)}$ in the real torus $(S^1)^2$ and the three semi-open holomorphic annulus $\mathscr{R}_j= \Log^{-1}(\mathcal{R}_j)$ for $j=1, 2, 3$. A complete description of $\overline{\Log^{-1}(v_0)}$ is given in \cite{NS-13}.  For any dimension,  it is  the same as  the complement in the real torus $(S^1)^n$ of an open zonotope (i.e. the coamoeba of a hyperplane).  In case where $n>2$, using  the description of the coamoeba of a hyperplane given in Theorem 3.3 \cite{NS-13} and the description of the $(n-1)$-dimensional pair-of-pants given in Proposition 2.24 \cite{M2-04}, one can check the  phase tropical hyperplane $(\mathbb{C}^*)^n$ is diffeomorphic to the complex projective space $\mathbb{CP}^{n-1}$ minus a tubular neirghborhood of the union $\mathscr{H}$ of $n+1$ hyperplanes in $\mathbb{CP}^{n-1}$. Let us be more explicite.

\vspace{0.1cm}

The hyperplane $H_\mathbb{K} := \{ (z_1,\ldots ,z_n)\in (\mathbb{K}^*)^n | \, z_1+\cdots+z_n+1=0  \}$ can be parametrized as follows:

\begin{displaymath}
 \left\{ \begin{array}{lll}
z_1(t) &=& t^{-x_1}e^{i\alpha_1}\\
z_2(t) &=& t^{-x_2}e^{i\alpha_2}\\
\vdots&\vdots&\vdots\\
z_{n-1}(t) &=& t^{-x_{n-1}}e^{i\alpha_{n-1}}\\
z_n(t)&=& 1-\sum_{j=1}^{j=n-1}t^{-x_j}e^{i\alpha_j}
\end{array} \right.
\end{displaymath}
with $x_j\in \mathbb{R}$ and $0\leq \alpha_j\leq 2\pi$ for $j=1,\ldots , n-1$. If we denote $\mathring{H}_t\subset (\mathbb{C}^*)^n$ the hyperplane  given by the parametrization for a fixed $t$. Then  all the family of hyperplanes $\{ \mathring{H}_t\}_{0<t\leq 1}$ is viewd as a single hyperplane in $(\mathbb{K}^*)^n$ and we  have $\mathring{H}_\infty = W(\mathring{H}_\mathbb{K})$  where $W$ is the map from $(\mathbb{K}^*)^n$ to $(\mathbb{C}^*)^n$ defined in Section 3. Also, the tropical hyperplane $\Gamma_n$ is the image by the logarithmic map of $\mathring{H}_\infty$. The following lemma gives a complete topological description of $\mathring{H}_\infty$.
\end{proof}

\begin{lemma}\label{description-of-H}
Let $\mathring{H}_\infty \subset (\mathbb{C}^*)^n$ be a phase tropical hyperplane and $\Gamma_n$ its image by the logarithmic map.
Then the inverse image of a point in the interior of an $l$-cell $\sigma\subset \Gamma_n$ is the product of a real $l$-torus with the coamoeba of a hyperplane in $(\mathbb{C}^*)^{n-l}$, i.e. if $x = (x_1,\ldots,x_n)\in \mathring{\sigma}$ then we have:
$$
\Log^{-1}(x) = (S^1)^l\times co\mathscr{A}(n-1-l),
$$
where $co\mathscr{A}(n-1-l)$ is the coamoeba a $(n-1-l)$-plane in $(\mathbb{C}^*)^{n-l}$.
\end{lemma}

\begin{proof}
Let $x$ be a point in the interior of an $l$-cell, then there exist $x_{j_1}, \ldots , x_{j_l}$ strictly negative and all the other $x_j$ are equal to zero. As $\mathring{H}_\infty$ is the limit when $t$ tends to zero (if we want t goes to infinity then we can make the change of variable in the parametrization,    $t$ by $\frac{1}{\tau}$), then for any fixed $\alpha_{j_1}, \ldots, \alpha_{j_l}$ we obtain the coamoeba of a hyperplane in $(\mathbb{C}^*)^{n-l}$  (recall that $\lim_{t\rightarrow 0}t^{-x_{j_u}} =0$, because $x_{j_u}<0$ for any $u=1,\ldots, l$). 
But $0\leq \alpha_{j_u}\leq 2\pi$, which means that the fiber over $x$
is the product of  the torus $(S^1)^l$ with the coamoeba of a hyperplane in $(\mathbb{C}^*)^{n-l}$. In particular, the inverse image of a $0$-cell is the coamoeba of a hyperplane in $(\mathbb{C}^*)^{n}$ which is equal to its phase limit set, and its topological description is given in \cite{NS-13}.
\end{proof}

Lemma \ref{description-of-H} gives a complete description of the phase tropical hyperplane $\mathring{H}_\infty$, which coincide with the description of a hyperplane in the projective space $\mathbb{CP}^n$ minus $n+1$  generic hyperplanes.

\vspace{0.2cm}

\subsection{A differential structure on phase tropical hypersurfaces} ${}$

\vspace{0.2cm}

In the general case,   let us denote by $\Gamma$ the  tropical variety limit of the family of  amoebas $\{\mathscr{A}_t\}$, where $\mathscr{A}_t$ is the amoeba of the variety $\mathring{V}_t$. Also, 
we assume that  the tropical hypersurface $\Gamma$ is smooth in the sense that every vertex of $\Gamma$ is dual to a simplex of Euclidean volume equal to $\frac{1}{n!}$. Therefore, locally for any vertex $v$ of $\Gamma$ there exists an open neighborhood ${U}_v$ diffeomorphic to the standard tropical hyperplane, in other words, tropical pair-of-pants. More precisely, there exists an affine linear transformation of $\mathbb R^n$ whose linear part ${A}_v$ belongs to $SL_n(\mathbb Z)$ such that ${U}_v$ is the image of the standard tropical hyperplane by ${}^tA_v^{-1}.$ Namely, $\overline{{U}}_v$ has $n+1$ boundary components isomorphic to an $(n-2)$-dimensional tropical hyperplane in $\mathbb R^{n-1}$ where $\mathbb R^{n-1}$ can be viewed as a boundary component of the tropical projective space $\mathbb{PT}^n$ represented by the standard simplex.

Let $v_1$ and $v_2$ be two adjacent vertices of $\Gamma$, in other words, there exists a compact edge $e$ with boundary $v_1$ and $v_2$. Then ${U}_{v_i}$ has a boundary component $\mathscr{B}_{ij}: = \partial_j{U}_{v_i}$ that can be viewed as a component of the boundary of a tubular neighborhood of a boundary component $\mathscr{B}_{ji}:=\partial_i{U}_{v_j}$. In other words, there exists an open neighborhood ${U}_{v_1v_2}$ of $v_1$ and $v_2$ containing ${U}_{v_1}$ and $\mathscr{U}_{v_2}$ such that ${U}_{v_1v_2}$ is the interior of the gluing of $\overline{{U}}_{v_1}$ and $\overline{{U}}_{v_2}$ along their boundaries $\mathscr{B}_{ij}$ and $\mathscr{B}_{ji}$ are joined by a vertical edge and all the other edges adjacent to $v_i ~(i=1,2)$ are horizontal (i.e., they are mutually parallel)  such that the reversing orientation diffeomorphism is given by $(z_1,\cdots,z_{n-1},z_n)\mapsto (z_1,\cdots,z_{n-1},\overline{z}_n)$. After gluing all pieces, we obtain a manifold $W_\infty(\Gamma)$ with boundary coming from unbounded 1-cells of $\Gamma_f$ where each unbounded 1-cell will corresponds to $\mathscr{B}_{ij}$ for some vertex $v_i$. Each $\mathscr{B}_{ij}$ is a circle fibration over a union of lower dimensional pair-of-pants $\mathscr{P}_{n-2}$ (see Figure 5).  We can remark that $W_\infty(\Gamma)$ is a topological description  of the decomposition of $\mathring{H}_\infty= W(\mathring{V}_{\mathbb{K}})$, where $\mathring{V}_{\mathbb{K}}$ is the hypersurface of $(\mathbb{K}^*)^n$ representing the family $\{\mathring{V}_t\}$. In other words, the family $\{\mathring{V}_t\}$ is viewed as a single hypersurface in  the algebraic torus $(\mathbb{K}^*)^n$.

\begin{figure}[h!]
\begin{center}
\includegraphics[angle=0,width=0.6\textwidth]{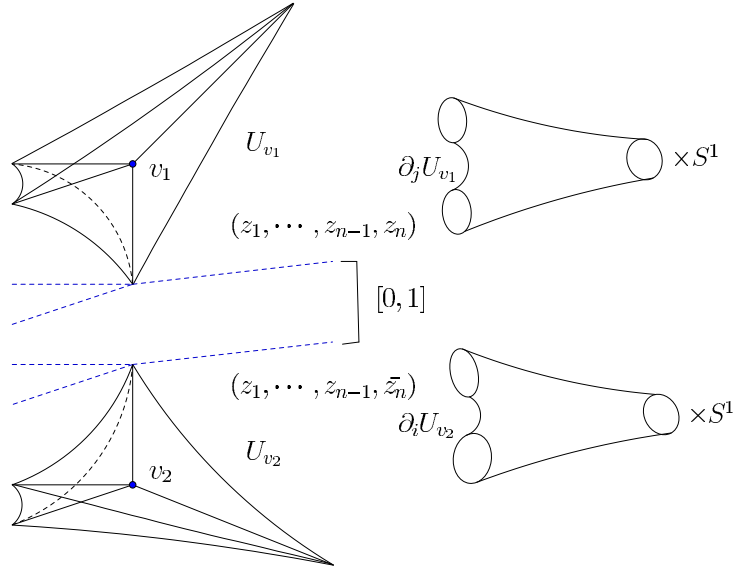}\quad
\caption{Gluing of two $2$-dimensional pairs-of-pants in $(\mathbb{C}^*)^3$ along one component of their boundaries.}
\label{}
\end{center}
\end{figure}

Let us denote by $M_\infty(\Gamma)$ the result of collapsing all fibers of these fibrations on  the boundary $\partial W_\infty(\Gamma)$ of $W_\infty(\Gamma)$. Then $M_\infty(\Gamma)$ is a smooth manifold. Indeed, this construction coincide locally with collapsing the boundary on $\overline{\mathscr{P}}_{n-1}$ which results in the projection space $\mathbb{CP}^{n-1}$ which is smooth.

\subsection{Proof of Theorem \ref{A}} ${}$

\vspace{0.2cm}

\noindent Since all smooth hypersurface with a fixed Newton polytope are isotopic, then we can choose any of them. More precisely, we will use for our subject the convenient one. Let $f_t(x)=\sum_{j\in \Delta\cap\mathbb{Z}^n}a_jt^{-v(j)}z^j$ be a polynomial with a parameter $t$, and $\mathring{V}_t=\{f_t=0\}\subset (\mathbb C^*)^n$. The family of $f_t$ can be viewed as a single polynomial in $\mathbb K[z_1^{\pm 1},\cdots,z_n^{\pm 1}]$. Therefore this family defines a hypersurface $\mathring{V}_{\mathbb K}\subset (\mathbb K^*)^n$. Let $\mathscr{A}_t:=\Log_t(\mathring{V}_t)$ and $\mathscr{A}_{\mathbb K}:=\Log_{\mathbb K}(\mathring{V}_{\mathbb K})$.

\noindent Let $\Gamma$ be a maximally   dual $\Delta$-complex  (i.e. all the element of  its dual the subdivision are simplex of Euclidean volume $\frac{1}{n!}$)  and $\nu:\Delta \cap \mathbb Z^n\to \mathbb R$ be the function such that $\Gamma=\Gamma_\nu$ i.e.,  $\Gamma_\nu$ is the tropical hypersurface defined by the tropical polynomial $\displaystyle{ \max_{\alpha\in \Delta \cap \mathbb Z^n } \{ \nu ({\alpha}) + <\alpha , x>  \}}$. 
Then we obtain a family of polynomial called a Viro-patchworking polynomial \cite{V-90}
$$
f_t(z)=\sum_{v\in \Delta\cap \mathbb Z^n}t^{-v(j)}z^j.
$$ 
Let us denote $\mathring{V}_t\subset (\mathbb C^*)^n$ the zero locus of the polynomial $f_t$. Using a foliation  of the amoeba of $\mathring{V}_t$  Mikhalkin obtains a map $\lambda_t=\pi_{\mathcal F_{\Gamma}}\circ \Log_t: V_t \to \Gamma$,  and proves in  Lemma 6.5, \cite{M2-04}  that $\mathring{V}_t$ is smooth for a sufficiently  large  $t\gg 0$.

\noindent First of all, $\Gamma$ looks locally as a tropical hyperplane after a  linear transformation with linear part  $SL_n(\mathbb Z)$. It means that $\Gamma$ can be locally identified to a tropical hyperplane in $\mathbb R^n$ by a linear transformation $F$ of $\mathbb R^n$ with a linear part in $SL_n(\mathbb Z)$.

\noindent It was shown in Lemma 6.5 \cite{M2-04}  that $V_t$ is also  smooth, and $\lambda_t$ satisfies a nice properties. Indeed, for $t\gg 0$, $\mathring{V}_t$ is smooth, and $\mathring{V}_t$ is an union of finite number of open sets, where each set is the image of a small perturbation of a hyperplane. Hence, its compactification $V_t\subset X_{\Delta}$ is smooth and transverse to the coordinate hyperplanes. Also, for a large $t\gg 0$, $V_t$ is isotopic to the variety $M_\infty(\Gamma)$ constructed above (which is a compactification of the phase tropical variety $\mathring{V}_\infty = W_\infty(\Gamma)$ the lifting of $\Gamma$ in $(\mathbb{C}^*)^n$), this comes from  Theorem 4 of Mikhalkin \cite{M2-04}, which proves the second statement of Theorem \ref{A}. This shows that $\mathring{V}_\infty$ is  also diffeomorphic to $\mathring{V}_t$ for sufficently  large $t\gg 0$ and the first statement of Theorem \ref{A} is proved.

\section{Construction of a natural symplectic structure on $\mathring{V}_{\infty}$}
Note that every pair-of-pants inherit a natural symplectic structure coming from  the one of the projective space $\mathbb{CP}^n$. Namely, the projective space $\mathbb{CP}^n$ is obtained from a closed pair-of-pants after collapsing its boundary. Indeed, each component of the boundary of a pair-of pants $\mathscr{P}^n$ is  a $S^1-$fibration over a lower dimensional pair-of-pants $\mathscr{P}^{n-1}$, and the result of collapsing all fibers of these $S^1-$fibrations is precisely the projective space $\mathbb{CP}^n$.

\vspace{0.2cm}

\subsection{Proof of Theorem \ref{B}} ${}$

\vspace{0.2cm}

Let $M_\infty(\Gamma)$  be the variety constructed in  Section 5, which  is a compactification of $\mathring{V}_\infty$ in the toric variety $X_\Delta$ where $\Delta$ is the degree of our original hypersurface $V$. The variety $M_\infty(\Gamma)$
  is obtained by gluing pairs-of-pants along a part of their boundary $\mathscr{B}_j$ that is a product of a holomorphic cylinder (i.e. an  annulus) in $\mathbb C^*$ with a lower dimensional  pair-of-pants $\mathscr{P}^{n-2}$  (i.e. along $[0,1]\times \mathscr{B}_j$).
 Moreover, each $\mathscr{B}_j$  is a circle fibration over $\mathscr{P}^{n-2}$, where the fibers are precisely the fibers of the annulus over the interval $[0,1]$:
$$
\mathscr{B}_j \longrightarrow \mathscr{P}^{n-2} \qquad \textrm{is an $S^1$-fibration}, 
$$
and
$$
\mathcal{A}\times\mathscr{P}^{n-2} = [0,1]\times \mathscr{B}_j \longrightarrow \mathscr{P}^{n-2} \qquad \textrm{is an annulus fibration},
$$
where $\mathcal{A}$ is the annulus $[0,1]\times S^1$.

Let us denote by $\omega_j^{(n-2)}$ the symplectic form on the pair-of-pants $\mathscr{P}^{n-2}$ coming from the projective space
$\mathbb{CP}^{n-2}$ and $ds\wedge dt$ the symplectic form on $S^1\times \mathbb{R}$. Hence, we obtain a symplectic form $\omega_j := ds\wedge dt + \omega_j^{(n-2)}$ on $[0,1]\times \mathscr{B}_j$. It means that we have a symplectic form on parts where the gluing was done.
Recall that $[0,1]\times \mathscr{B}_j$ can be seen as a neighborhood of a boundary component of the pair-of-pants $\mathscr{P}^{n-1}$. On the other part of $\mathscr{P}^{n-1}$ i.e., $\mathscr{P}^{n-1}\setminus \cup_j([0,1]\times\mathscr{B}_j)$, we already have the symplectic form of a pair-of-pants $\omega^{(n)}$ and the pull back of $\omega^{(n)}$ on the factor $\mathscr{P}^{n-2}$ of any boundary component  is precisely $\omega_j^{(n-2)}$.
However, when we glue $[0,1]\times \mathscr{B}_j$ and $[0,1]\times \mathscr{B}_i$ where the first part  is equipped with the form $ds\wedge dt + \omega_j^{n-2}$ then the second should be equipped with the form $-ds\wedge dt + \omega_i^{n-2}$ because the gluing was done with a reversing orientation (recall  that the forms $\omega_j^{(n-2)}$ and $\omega_i^{(n-2)}$ are the same).
On the other hand, the symplectic forms outside of the gluing parts are well defined since each component is symplectically an open pair-of-pants which is a hyperplane in  the complex algebraic torus $(\mathbb{C}^*)^n$. After taking the compactification of such hyperplanes in the projective space $\mathbb{CP}^{n-1}$, the restriction of these forms on the infinite parts (i.e. the $\mathbb{CP}^{n-2}$'s) are precisely the forms $\omega_j^{n-2}$'s. This gives rise to a global symplectic form  $\mathring{\omega}_{nat}$ on  $\mathring{V}_\infty$. This proves that $\mathring{V}_\infty$ has a natural symplectic structure because all the forms that we used are constructed naturally and the first part of Theorem \ref{B} is proved.

\vspace{0.1cm}


Let us denote by $\mathring{\omega}_t = \iota_t^*(\omega)$  the symplectic form on $\mathring{V}_t$ where $\iota_t$ is the inclusion of $\mathring{V}_t$ in the complex algebraic torus $(\mathbb{C}^*)^n$, $\omega$ is the symplectic form on $(\mathbb{C}^*)^n$ defined by (1).  Using Moser's trick, Mikhalkin showed that $M_\infty(\Gamma)$ is symplectomorphic
to  $V_t$ for a sufficiently large $t\gg0$. Let us denote this symplectomorphism by $\phi$. Hence we have the following commutative digram:

\begin{equation}\label{(2)}
\xymatrix{
(\mathring{V}_\infty, \mathring{\omega}_{nat})\ar[rr]^{\psi = \phi_{|\mathring{V}_\infty}}\ar@{^{(}->}[d]_{j}&&
(\mathring{V}_t, \mathring{\omega}_{t})\ar@{^{(}->}[d]^{i}\cr
(M_\infty(\Gamma), \omega_{nat})\ar[rr]^{\phi}&&(V_t, \omega_t).
}
\end{equation}
This means that  for a sufficiently large $t\gg0$, \, $\mathring{V}_\infty$ is also symplectomorphic to  $\mathring{V}_t$, and the second statement Theorem \ref{B} is proved. Recall that we can prove Theorem \ref{B} using a generalization of Moser's trick for non compact manifolds proved by R. E. Greene and K. Shiohama on 1979 in \cite{GS-79}.

\end{document}